\documentclass[10pt,reqno]{amsart}
\usepackage{amssymb}
\usepackage{faktor}
\usepackage{fullpage}
\usepackage{amsthm}
\usepackage{colonequals}
\usepackage{bm}
\usepackage{calrsfs}
\usepackage{latexsym}
\usepackage{float}
\restylefloat{table}
\usepackage{hyperref}
\usepackage{amsmath}
\usepackage{mathrsfs}
\usepackage[font={small,it}]{caption}
\usepackage{amscd}
\usepackage[all,cmtip]{xy}
\usepackage[usenames]{color}
\usepackage{graphicx}
\usepackage{color}
\usepackage{amscd}
\usepackage{float}
\usepackage{graphics}
\usepackage{tikz}
\usepackage{tikz-cd}
\usepackage{comment}
\usepackage{xspace}
\usepackage{mathtools}

\usepackage{amssymb}
\usepackage{amsthm}
\usepackage{graphicx}
\usepackage{subfig}
\usepackage{enumerate}
\usepackage{hyperref}

\usetikzlibrary{patterns,decorations.pathreplacing}
\usepackage{marginnote}

\theoremstyle{plain}

\DeclareMathAlphabet{\pazocal}{OMS}{zplm}{m}{n}

\newtheorem{theorem}{Theorem}[section]
\newtheorem{proposition}[theorem]{Proposition}
\newtheorem{lemma}[theorem]{Lemma}
\newtheorem{corollary}[theorem]{Corollary}

\newtheorem{remark}[theorem]{Remark}

\DeclareFontFamily{U}{wncy}{}
    \DeclareFontShape{U}{wncy}{m}{n}{<->wncyr10}{}
    \DeclareSymbolFont{mcy}{U}{wncy}{m}{n}
    \DeclareMathSymbol{\Sh}{\mathord}{mcy}{"58}
\theoremstyle{definition}

\newcommand{\appsection}[1]{\let\oldthesection\thesection
\renewcommand{\thesection}{Appendix \oldthesection}
\section{#1}\let\thesection\oldthesection}
\newtheorem{definition}[theorem]{Definition}

\theoremstyle{remark}

\DeclareMathOperator{\Div}{Div}

\DeclareMathOperator{\Norm}{Norm}

\def\Z{{\mathbb{Z}}}

\def\Q{{\mathbb{Q}}}

\def\P{{\mathbb{P}}}

\def\O{{\mathcal{O}}}

\def\L{{\mathcal{L}}}
\def\T{{\mathcal{T}}}

\DeclareMathOperator{\rk}{rank}

\usepackage{microtype}

\DeclareMathOperator{\Pic}{Pic}

\DeclareMathOperator{\NS}{NS}

\newcommand{\jerson}[1]{{\color{blue} ($\spadesuit$ Comment: #1)}}

\title{Picard rank and Ulrich line bundles on bidouble planes}

\author{Jerson Caro}
\address{Jerson Caro \\
Departamento de Matemáticas, Universidad de los Andes, Bogotá, Colombia
}
\email{jl.caro10@uniandes.edu.co}

\author {Juan Cruz-Penagos}
\address{Universidad del Valle Departamento de Matemáticas, Calle 13 100-00, Cali-Colombia}
\email[J. Cruz-Penagos]{juan.sebastian.cruz@correounivalle.edu.co}

\author{Sergio Troncoso}
\address{Politecnico di Torino Corso Duca degli Abruzzi 24
10129, Torino, Italy}
\email[S. Troncoso]{sergio.troncoso@polito.it}

\thanks{J.C. was supported by a Simons Foundation grant (Grant \#550023). J.C-P.  was supported by the Master's Program, Universidad del Valle. S.T. was partially supported by the PRIN project “Multilinear Algebraic Geometry” of MUR (2022ZRRL4C)}
\begin{document}

\begin{abstract}
We determine the Picard number and the Ulrich complexity of general bidouble covers of the projective plane, providing the first systematic study of Ulrich bundles on non-cyclic abelian covers. For a bidouble plane branched along three smooth curves of degrees $n_1,n_2,n_3$, we show that 
$\rho(S)=1$ unless $(n_1,n_2,n_3)$ belongs to an explicit list, thereby extending Buium’s classical results on double planes to the non-cyclic case. As an application, we determine the range of branch degrees for which Ulrich line bundles could exist. Our method combines the invariant-theoretic decomposition of 
$H^2(S,\Q)$ under the Galois group with cohomological criteria for Ulrich bundles.
\end{abstract}
\maketitle


\section{Introduction}

Ulrich bundles constitute a distinguished class of arithmetically Cohen–Macaulay vector bundles whose
cohomological extremality reflects the embedding of a projective variety into projective space. Since their introduction by Eisenbud and Schreyer \cite{ESW03}, Ulrich bundles have played a central role in algebraic geometry, appearing in the study of syzygies, Chow forms, and the geometry of linear projections (see Section \ref{UBundles} for the definition and basic facts).  

Despite intensive progress, the existence and structure of Ulrich bundles on surfaces of general type are currently understood only in a limited number of cases. The case of cyclic covers
of $\mathbb P^2$ has been particularly fruitful. Indeed, for double planes, Ulrich line bundles were studied in \cite{PN21}, while rank two Ulrich bundles were constructed in \cite{ST22, MKNP25}. More generally, existence results are now known for double coverings of projective spaces, and on Fano and Calabi--Yau double coverings of $\mathbb{P}^3$ \cite{MKNP25, Vac25, Vac251} and for cyclic coverings of higher rank \cite{PP25}. It is also worth mentioning the work of \cite{BHMP}, where the authors investigate aCM sheaves on double planes, that is, hyperplanes of $\P^3$ with multiplicity two.

This paper provides the first systematic study of the Ulrich complexity of bidouble planes with respect to the pullback of a line. Although this divisor is not very ample, it is ample and globally generated, so the definition of Ulrich bundles extends naturally (see for example \cite{But24}). Bidouble planes are understood here as $(\mathbb{Z}/2\mathbb{Z})^2$-covers of $\mathbb{P}^2$ branched along three transversal curves whose degrees share the same parity. We call the bidouble plane \textit{even} if all of these degrees are even; otherwise, we call it \textit{odd}.

Our first main result determines the Picard number of a general bidouble plane,
extending the classical results of Buium \cite{Bui83}
from double planes to non-cyclic abelian covers.

\begin{theorem}\label{Main Theorem}
Let $S\to \mathbb{P}^2$ be a smooth bidouble plane branched along general curves $D_1,D_2,D_3$ of degrees $n_1,n_2,n_3$, respectively. Then $\rho(S)>1$ if and only if some permutation of $(n_1,n_2,n_3)$ belongs to
\[
\{(0,2,2n):n\geq 1\}\;\cup\;\{(0,4,2n):n\geq 2\}\;\cup\;\{(1,3,2n+1):n\geq 0\}\;\cup\;\{(2,2,2n):n\geq 1\}.
\]
\end{theorem}

The computation relies on the eigenspace decomposition of $H^2(S,\mathbb Q)$
under the $(\mathbb Z/2)^2$-action and on a detailed analysis of the Hodge
structure associated to the intermediate double covers determined by the three
involutions of $S$. One consequence is that, for general choices of branch curves, the
surface $S$ carries no Ulrich line bundles with respect to $\mathcal{O}_S(1)$: the Néron--Severi group is too small for such a
bundle to exist (cf. Theorem \ref{line bundles L}). The smallest integer $r\geq 1$ such that there exists a rank $r$ Ulrich bundle on $S$ with respect to $H$ is called the \emph{Ulrich complexity} and is denoted by $\operatorname{uc}(S, H)$.
Our second main theorem establishes lower bounds on the Ulrich complexity of bidouble planes.

\begin{theorem}\label{Theorem B}
Let $S$ be a smooth bidouble plane of $\mathbb{P}^2$ branched along generic smooth curves of degrees $n_1, n_2, n_3$, and assume $n_1 \le n_2 \le n_3$. Define
\[
T_1 = \{(0,4,2n)\mid n\ge 2\}\;\cup\;\{(2,2,2n)\mid n\ge 1\}, 
\qquad\text{and}\qquad
T_2 = \{(0,2,2),\, (0,2,4)\}.
\]
If $S$ is an odd bidouble cover, then its Ulrich complexity with respect to $\mathcal{O}_S(1)$ satisfies $\operatorname{uc}(S,\mathcal{O}_S(1))>1$. 

\noindent If $S$ is an even bidouble cover, then:
\begin{itemize}
    \item[(a)] If $(n_1,n_2,n_3)\notin T_1\cup T_2$, then $\operatorname{uc}(S,\mathcal{O}_S(1)) > 1$.
    \item[(b)] If $(n_1,n_2,n_3)\in T_2$, then
    $\operatorname{uc}(S,\mathcal{O}_S(1)) = 1$.
\end{itemize}
\end{theorem}

The paper is organized as follows. In Section \ref{section2}, we recall the basic facts about Ulrich bundles and bidouble covers. In Section \ref{section3}, we study the Nerón--Severi group of bidouble planes: first, we show that the Picard number of a bidouble plane is determined by those of its intermediate covers; second, we establish the numerical conditions under which the Picard number is strictly larger than one, leading to the proof of Theorem \ref{Main Theorem}. In Section \ref{section4},
we show that the existence of an Ulrich line bundle forces the Picard number to be greater than one. For double planes, it was shown in \cite{PN21} that the only obstruction is having Picard number one.  
In contrast, we exhibit families of bidouble covers with Picard number greater than one that nevertheless admit no Ulrich line bundles. On the positive side, we also construct explicit Ulrich line bundles in the cases $(n_1,n_2,n_3)=(0,2,2)$ and $(0,2,4)$.

\begin{remark}
We restrict to bidouble covers of $\mathbb{P}^2$ for the following reasons.
Most of our non-existence results for Ulrich line bundles hinge on showing that the Picard number is one, which in turn forces the base surface to have a Picard number one. Although some arguments extend to more general bases, proving that the intermediate quotients also have Picard number one (cf.~Lemma~\ref{NS n1+n2>6}) requires realizing them as hypersurfaces in a weighted projective space. This situation arises precisely when the base is $\mathbb{P}^2$. Notice that one of our results (cf.~Proposition~\ref{(0,2,2n)}) establishes the nonexistence of Ulrich line bundles on double covers of $\mathbb{P}^1 \times \mathbb{P}^1$ branched along symmetric divisors, a fact that has not been previously recorded in the literature.
\end{remark}


\section*{Acknowledgements} We thank Pedro Montero, Joan Pons-Llopis, and Roberto Villaflor for useful discussions. In addition, we thank Jorge Duque, Francesco Malaspina, Johannes Rau, David Rohrlich, and Giancarlo Urz\'ua for their comments on a previous version of this manuscript.
The first author was supported by Simons Foundation grant no.\ 550023, the second author was partially supported by the Master's Program, Universidad del Valle, and the third author was partially supported by the PRIN project “Multilinear Algebraic Geometry” of MUR (2022ZRRL4C). The third author also extends thanks to Rijksuniversiteit Groningen/Campus Fryslân for their hospitality. 


\section{Preliminaries}\label{section2}

All varieties in this work are defined over the complex numbers $\mathbb{C}$.

\subsection{Ulrich bundles}\label{UBundles}

In recent years, considerable attention has been devoted to a class of vector bundles known as Ulrich bundles. For a smooth projective variety with an embedding $X \subset \mathbb{P}(|H|)$ given by a very ample divisor $H$, a vector bundle $\mathcal{E}$ on $X$ is called Ulrich with respect to $H$ if
$
H^i(X, \mathcal{E}(-jH)) = 0 $ for all $i \ge 0$ and $1 \le j \le \dim X$.
 We refer the reader to \cite{joan} for a comprehensive introduction to Ulrich bundles.

In this paper, we adopt a slightly more general perspective: rather than assuming that the polarization is very ample, we only require it to be ample and globally generated.

\begin{definition}\label{Ulrichdef}
Let $X$ be a smooth projective variety of dimension $n \ge 1$, and let $H$ be an ample and globally generated line bundle on $X$. A vector bundle $\mathcal{E}$ on $X$ is said to be Ulrich with respect to $H$ if
\[
H^i(X, \mathcal{E}(-jH)) = 0,
\]
for all $i \ge 0$ and $1 \le j \le n$.
\end{definition}

Some considerations justify this more general framework. Firstly, numerous essential properties of classical Ulrich bundles remain valid under these assumptions. Secondly, bidouble covers of $\mathbb{P}^2$ naturally possess polarizations that are ample and globally generated, but not necessarily very ample. For these reasons and following the convention in \cite[Section 1]{Takao}, we call a pair $(X, H)$ a \textit{polarized variety} if $X$ is a smooth variety and $H$ is an ample and globally generated line bundle.

 \begin{definition}\label{complexity}
    The smallest $r\geq 1$ such that there is a rank $r$ Ulrich bundle on the polarized variety $(X, H)$ is called the \textit{Ulrich complexity} of $X$, denoted $\operatorname{uc}(X, H)$.
\end{definition}

When focusing on surfaces, the Ulrich condition of a rank $r$ vector bundle $\mathcal{E}$ reduces to two cohomological vanishing together with two conditions on the Chern classes of $\mathcal{E}$, as follows. 

\begin{proposition}[{\cite[Proposition 2.1]{Casnon}}]\label{iff Ulrich}
Let $(S,H)$ be a smooth polarized  surface. For a vector bundle $\mathcal{E}$ on $S$, the following assertions are equivalent:
\begin{enumerate}
    \item $\mathcal{E}$ is an Ulrich bundle,
    \item $\mathcal{E}^{\vee}(3H+K_S)$ is an Ulrich bundle,
    \item  $\mathcal{E}$ is an aCM bundle and
   \begin{align}
c_1(\mathcal{E})\cdot H 
   &= \tfrac{\operatorname{rk}(\mathcal{E})}{2}(3H+K_S)\cdot H,\label{ulrich-equalities1}\\
c_2(\mathcal{E}) 
   &= \tfrac{1}{2}\big(c_1^2(\mathcal{E}) - c_1(\mathcal{E})\cdot  K_S\big)
   - \operatorname{rk}(\mathcal{E})(H^2 - \chi(\mathcal{O}_S)),\label{ulrich-equalities2}
   \end{align}
    \item $h^0(S,\mathcal{E}(-H))=h^0(S,\mathcal{E}^{\vee}(2H+K_S))=0$ and Equalities \eqref{ulrich-equalities1} and \eqref{ulrich-equalities2} hold.
\end{enumerate}
\end{proposition}

\subsection{Bidouble covers of surfaces}

We provide only a brief exposition on bidouble covers on surfaces. For a comprehensive treatment, we refer the reader to \cite{Cat84, PAC}.

Let $\pi:X\to Y$ be a finite cover of $Y$ with Galois group $G\coloneqq\operatorname{Gal}(\mathbb{C}(X)/\mathbb{C}(Y))\cong(\mathbb{Z}/2\mathbb{Z})^2$.
Such a map is called a \textit{bidouble cover} of $Y$. If $X$ is normal, then 
\[
\pi_*\mathcal{O}_X\cong \mathcal{O}_Y\oplus\bigoplus_{i=1}^3\mathcal{L}_{\sigma_i},
\]
where $\sigma_1$ and $\sigma_2$ are generators of $G$ and $\sigma_3=\sigma_1\sigma_2$, and each $\mathcal{L}_{\sigma_i}$ is a line bundle on $Y$ determined by $\sigma_i$. Thus,
$X=\operatorname{Spec}_{\mathcal{O}_Y}\left(\pi_*\mathcal{O}_X\right)$,
where the structure of $\mathcal{O}_{S}$-algebra is given by maps $\mathcal{L}_{\sigma_i}\oplus \mathcal{L}_{\sigma_j}\to \mathcal{L}_{\sigma_i+\sigma_j}$.

\begin{proposition}[{\cite[Proposition 2.3]{Cat84}}]\label{Pcatanese}
Let $\pi\colon X\to Y$ be a smooth bidouble cover. Then $\pi$ is uniquely determined by six effective divisors $D_1,D_2,D_3$, and $L_1,L_2,L_3$ on $Y$ such that 
\begin{enumerate}
    \item $2L_i\equiv D_j+D_k$, and
    \item $L_i+D_i\equiv L_j+L_k$,   
\end{enumerate}
for every permutation $(i,j,k)$ of $(1,2,3)$.

\noindent Moreover, the branch locus of $\pi$ is $D=\bigcup_{i=1}^3D_i$, which must be a normal crossing divisor on $Y$.
\end{proposition}

\begin{definition}\label{even}
A \emph{smooth bidouble plane} is a smooth bidouble cover $\pi \colon X \to \P^{2}$,
ramified along three smooth curves $D_{1},D_{2},D_{3}$ of degrees $n_{1},n_{2},n_{3}$.  
By Proposition~\ref{Pcatanese}, the relations between the divisors $D_{i}$ and the line bundles $L_{i}$ imply that the integers $n_{1},n_{2},n_{3}$ have the same parity.  

We say that $X$ is \emph{even} (respectively, \emph{odd}) if each $n_{i}$ is even (respectively, $n_{i}$ is odd).
\end{definition}

The following proposition describes several invariants associated with a bidouble plane. A detailed proof, as well as a more general formulation, can be found in \cite{Cat84}.

\begin{proposition}\label{invariants}
Let $\pi: X\to \mathbb{P}^2$ be a smooth bidouble plane defined as in Proposition \ref{Pcatanese} with $\operatorname{deg}(D_i)=n_i$. Then
\begin{enumerate}
\item  $\pi_*\mathcal{O}_S\cong \mathcal{O}_{\mathbb{P}^2}\oplus\mathcal{O}_{\mathbb{P}^2}{(\frac{-n_2-n_3}{2})}\oplus\mathcal{O}_{\mathbb{P}^2}{(\frac{-n_1-n_3}{2})}\oplus\mathcal{O}_{\mathbb{P}^2}{(\frac{-n_1-n_2}{2})}$,
    \item $K_S\equiv \pi^*(\mathcal{O}_{\mathbb{P}^2}(-3))+R$, where $2R=\pi^*(\mathcal{O}_{\mathbb{P}^2}(n_1+n_2+n_3))$, and $K_S^2=(n_1+n_2+n_3-6)^2$,
    \item  $q(S)=0$,
    \item $\chi(\O_S)=4+\frac{1}{4}(\sum_{i=1}^{3}n^{2}_i+n_1n_2+n_1n_3+n_2n_3-6\sum_{i=1}^{3}n_i)$.
\end{enumerate}  
\end{proposition}

\begin{remark}
\cite[Lemma~2.2]{Gar22} shows that, for a smooth bidouble plane $\pi\colon X \to \mathbb{P}^2$, at most one of the branch divisors $D_i$ can be trivial; otherwise the surface $S$ would be disconnected.
\end{remark}



\section{Nerón--Severi group of general bidouble planes}\label{section3} 
This section aims to compute the Picard number of a bidouble cover. 
A central feature of the construction of Catanese \cite[Section 2]{Cat84} is the existence of three intermediate quotients $Y_1, Y_2, Y_3$, corresponding to the quotients by the three index-two subgroups of $G$. Concretely, for a bidouble cover $\pi\colon S\to Y$ one has the commutative diagram
\begin{equation}\label{diagram}
\xymatrix{
&S \ar@{->}[dl]_{p_3}\ar@{->}[dr]^{p_{1}} \ar@{->}[d]^{p_{2}}&\\
Y_3 \ar@{->}[dr]_{\pi_3} & Y_2 \ar@{->}[d]^{\pi_2} & Y_1  \ar@{->}[dl]^{\pi_1}   \\
& Y, &}    
\end{equation}
where all the arrows are double covers. The intermediate quotient $Y_i$ is branched along $D_j+D_k$ whenever $(i,j,k)$ is a permutation of $(1,2,3)$.

The lemma below relates the Picard number of $S$ to the Picard numbers of the intermediate quotients. Its proof closely follows \cite[Corollary 3.17]{Gar22}, which we restate here for clarity.
\begin{lemma}\label{NS}
Let $\pi\colon  S\to \P^2$ be a smooth bidouble plane. Suppose that $\rho(Y_2)=\rho(Y_3)=1$, then $\rho(S)=\rho(Y_1)$. 
\end{lemma}

\begin{proof}
Let $\sigma_1$ and $\sigma_2$ be generators of $G$ and let $\sigma_3=\sigma_1\sigma_2$. Let $Y_i$ denote the quotient surface $S/\langle \sigma_i\rangle$.
We can decompose $H^2(S,\Q)$ into four eigenspaces corresponding to the characters of $G$:
\[
H^2(S,\mathbb{Q}) = H^2(S,\Q)^{++} \oplus H^2(S,\Q)^{+-} \oplus H^2(S,\Q)^{-+} \oplus H^2(S,\Q)^{--}.
\]
Here, $H^2(S,\Q)^{\varepsilon_1 \varepsilon_2}$ is the subspace where $\sigma_i$ acts by $\varepsilon_i \in \{\pm1\}$. Notice that $H^2(\P^2,\Q)\cong H^2(S,\Q)^{++}$ and due to the fact that 
$H^2(S/\langle \sigma_i\rangle,\Q)\cong H^2(S,\Q)^{\langle \sigma_i\rangle}$, we have the following isomorphisms:
\begin{align}
H^2(Y_1,\mathbb{Q}) &\cong H^2(S,\Q)^{++} \oplus H^2(S,\Q)^{+-},\nonumber\\
H^2(Y_2,\mathbb{Q}) &\cong H^2(S,\Q)^{++} \oplus H^2(S,\Q)^{-+},\\
H^2(Y_3,\mathbb{Q}) &\cong H^2(S,\Q)^{++} \oplus H^2(S,\Q)^{--}. \nonumber  
\end{align}
Then, we have
\[
H^2(S,\mathbb{Q}) = H^2(Y_1,\Q) \oplus H^2(S,\Q)^{-+} \oplus H^2(S,\Q)^{--}.
\]
Now, we pass to algebraic classes. Since $S$ is smooth projective, the Néron–Severi
group $\NS(S)$ is a Hodge substructure of $H^2(S,\Q)$ stable under $G$.
Thus, it decomposes accordingly:
\[
\NS(S)\otimes\Q = (\NS(S)^{++}\oplus \NS(S)^{+-}\oplus \NS(S)^{-+}\oplus \NS(S)^{--})\otimes \Q.
\]
For a finite double cover $p_i:S\to Y_i$, the pullback induces an injection
$p_i^*:\NS(Y_i)\otimes\Q\hookrightarrow \NS(S)\otimes\Q$ with image exactly the
$\sigma_i$-invariant part $\NS(S)^{\sigma_i=+1}$; conversely, any divisor class on
$S$ fixed by $\sigma_i$ descends to $Y_i$. Therefore
\begin{align}\label{NS Yi}
\NS(Y_1)\otimes\Q &\cong (\NS(S)^{++}\oplus \NS(S)^{+-})\otimes \Q,\nonumber\\
\NS(Y_2)\otimes\Q &\cong (\NS(S)^{++}\oplus \NS(S)^{-+})\otimes \Q,\\
\NS(Y_3)\otimes\Q &\cong (\NS(S)^{++}\oplus \NS(S)^{--})\otimes \Q.\nonumber    
\end{align}

Since $\NS(\P^2)=\Z \cdot H$,
$\NS(S)^{++}\cong \pi^*\NS(\P^2)$ has rank $1$.
By hypothesis $\rho(Y_2)=1$ and $\rho(Y_3)=1$. Comparing ranks in $\eqref{NS Yi}$ gives
\[
\rho(Y_2)=\rk\NS(S)^{++}+\rk\NS(S)^{-+}=1,
\]
which implies $\rk\NS(S)^{-+}=0$ and analogously one can prove that $\rk\NS(S)^{--}=0$.
Hence
\[
\NS(S)\otimes\Q  =  (\NS(S)^{++}\oplus \NS(S)^{+-})\otimes \Q.
\]
Using $\eqref{NS Yi}$ we obtain
\[
\NS(Y_1)\otimes\Q  \cong  (\NS(S)^{++}\oplus \NS(S)^{+-})\otimes \Q
 \cong  \NS(S)\otimes\Q,
\]
so $\rho(Y_1)=\rho(S)$, which gives the desired result.
\end{proof}
\begin{remark}\label{equal rho}
If $S$ and $Y_1$ are smooth and $\rho(Y_1)=\rho(S)$, then for any line bundle 
$\mathcal{L}$ on $S$ we have 
$\pi^*(\Norm(\mathcal{L})) \equiv 2\mathcal{L}$. In particular, every element of $\NS(S)/\pi^{*}\NS(Y_1)$ has order $\le2$.
The line bundle $\Norm(\mathcal{L})$ on $Y_1$, called the \emph{norm} of $\mathcal{L}$, 
is defined as follows: for each point $y \in Y_1$, the fiber 
$\Norm(\mathcal{L})_y$ is the tensor product 
$\bigotimes_{x \in \pi^{-1}(y)} \mathcal{L}_x^{\otimes n_{\pi,x}}$,
where $n_{\pi,x}$ denotes the multiplicity of $\pi$ at the point $x$.
\end{remark}
The corollary below follows directly from the proof of Lemma \ref{NS}. 
\begin{corollary}\label{Cor Picard}
Let $\pi: S\to \P^2$ be a smooth bidouble plane. Then $\rho(S)=1$ if and only if
$$\rho(Y_1)=\rho(Y_2)=\rho(Y_3)=1.$$
\end{corollary}

The following lemma and proposition provide sufficient conditions for the intermediate quotients to have minimal Picard number.

\begin{lemma}\label{NS n1+n2>6}
Let $\pi\colon Y\to \P^2$ be the double cover branched along $D = D_1 + D_2$, where $D_1$ and $D_2$ are general smooth plane curves of degrees $n_1$ and $n_2$, respectively, meeting transversely, and suppose that $n_1 + n_2 \ge 6$. Let $\nu: \widetilde Y \to Y$ be the minimal resolution. Then
$\rho(\widetilde Y) = 1 + n_1 n_2$,
and hence
$\rho(Y) = 1$.

\end{lemma}

\begin{proof}
Let $H$ be the class of a line in $\mathbb P^2$, and write $2d = n_1 + n_2 \ge 6$.
Let $f_i\in H^0(\mathbb P^2,\mathcal O_{\mathbb P^2}(n_i))$ be homogeneous equations defining $D_i$, and set
\[
Y \;=\; \{\, w^2 = f_1(x,y,z)f_2(x,y,z) \,\} \subset \mathbb P(1,1,1,d).
\]
For a general choice of $(f_1,f_2)$, the curves $D_1$ and $D_2$ meet transversely in $n_1n_2$ points and have no other singularities. Consequently, $Y$ has exactly $n_1 n_2$ rational double points (of type $A_1$), one over each point $p\in D_1\cap D_2$. Denote by $E_p$ the exceptional $(-2)$--curve above $p$ in the minimal resolution $\nu:\widetilde Y\to Y$.

The classes $(\pi\circ\nu)^*H$ and the $E_p$ are algebraic on $\widetilde{Y}$. Let $L$ be the lattice generated by $(\pi\circ\nu)^*H$ and the $E_p$ for every $p\in D_1\cap D_2$.
These classes are linearly independent, thus $\rho(\widetilde Y)\ge 1+n_1 n_2$.
To compute the Picard numbers of $\widetilde Y$ and $Y$, consider a one-parameter family $f_t\in H^0(\mathbb P^2,\mathcal O_{\mathbb P^2}(2d))$ with $f_0=f_1f_2$ and such that for $t\neq0$, the surface $Y_t=\{w^2=f_t\}$ is smooth. Define
\[
\mathcal Y = \{\, w^2 = f_t(x,y,z) \,\}
\;\subset\; \mathbb P(1,1,1,d)\times \Delta,
\]
where $\Delta$ is a small analytic disk around $0\in\mathbb C$. 
Let $\pi_\Delta:\mathcal Y\to\Delta$ be the projection to the parameter. 
Then $\mathcal Y\to\Delta$ is flat, with central fiber $Y_0=Y$ and smooth fibers $\pi_t\colon Y_t\to \P^2$ for $t\ne0$.

By the semistable reduction theorem (see \cite[p. 102]{Morrison84}), 
there exists a finite base change $\Delta'\to\Delta$ (ramified only over $0$) and a birational modification of the total space 
such that the induced family $\mathcal Y'\to\Delta'$ is semistable, that is, its central fiber is a normal crossings divisor with smooth components.
Replacing $\mathcal Y$ by this semistable model ensures that the 
Clemens–Schmid sequence and the specialization morphism of mixed Hodge structures apply functorially. 
In the semistable model, besides the main component $\widetilde{Y}$, the central fiber contains, for each node 
$p \in D_1\cap D_2$, an additional
smooth rational surface $F_p$, obtained as the minimal resolution of the corresponding local double point degeneration.  Moreover, each $F_p$ meets $\widetilde Y$ transversely along the
exceptional curve $E_p$.

By Buium \cite[Theorem, p.~1]{Bui83}, a general double plane branched along a smooth divisor of degree at least six has Picard number $1$. Hence, the orthogonal complement (with respect to the cup product) of $\pi_t^*H$ in $H^2(Y_t,\mathbb{Q})$, denoted $P_t$, satisfies $P_t \cap H^{1,1}(Y_t) = 0$ for $t$ sufficiently close to $0$.

Now let $L_\mathbb Q=L\otimes\mathbb Q$ and let $P$ denote the orthogonal complement of $L_\mathbb Q$ in $H^2(\widetilde Y,\mathbb Q)$. We claim that $P\cap H^{1,1}(\widetilde Y)=0$. For the sake of contradiction, suppose that there exists a nonzero class $\alpha\in P\cap H^{1,1}(\widetilde Y)$. Let
$\mathrm{sp}: H^2(\widetilde Y,\mathbb Q) \longrightarrow H^2_{\lim}(Y_t,\mathbb Q)
$ be the specialization morphism of mixed Hodge structures associated to the semistable family. By the Local Invariant Cycle Theorem (see \cite[p.~108]{Morrison84} and \cite[Th\'eor\`eme 4.1.1]{Deligne71}), the image of $\mathrm{sp}$ is the monodromy--invariant subspace. Hence $\mathrm{sp}(\alpha)$ is a monodromy--invariant class in $H^{2}(Y_t,\Q)\cap H^{1,1}(Y_t)$ for $t$ sufficiently close to $0$.

The kernel of $\mathrm{sp}$ is generated by the vanishing cycles, which on $t=0$ correspond to the classes of the exceptional curves $E_p$ by the Clemens-Schmid exact sequence (see \cite[p.~108]{Morrison84}). Since $\alpha\in P$ is orthogonal to each $E_p$, we have $\mathrm{sp}(\alpha)\neq0$. 

The hyperplane class $\pi_t^*H$ is the restriction of $\O_{\P(1,1,1,d)}(1)$ to $Y_t$. Therefore,
$$
\langle \mathrm{sp}(\alpha), \pi_t^*H \rangle
= \langle \alpha,(\pi\circ\nu)^*H\rangle = 0.
$$
Thus $\mathrm{sp}(\alpha)\in P_t\cap H^{1,1}(Y_t)$, which is a contradiction as $\mathrm{sp}(\alpha)\neq0$. This contradiction proves that $P\cap H^{1,1}(\widetilde Y)=0$ as claimed. Therefore $\mathrm{NS}(\widetilde Y)\otimes\mathbb Q = L_\mathbb Q$, which implies that $\rho(\widetilde Y)=1+n_1 n_2$.
Finally, since $Y$ has only rational double points, the natural map $\mathrm{Pic}(\widetilde Y)\to\mathrm{Pic}(Y)$ is surjective with kernel generated by the exceptional curves $\{E_p\}$. Consequently
$\rho(Y)=\rho(\widetilde Y)-n_1 n_2 = 1$,
which yields the desired result.
\end{proof}

\begin{proposition}
\label{Picard intermediate}
Let $Y$ be a double cover of $\P^2$ branched along $D=D_1+D_2$ where $D_1$ and $D_2$ are general effective divisors of degrees $n_1$ and $n_2$, respectively, and assume that $n_1\le n_2$. Then $\rho(Y)>1$ if and only if $(n_1,n_2)\in\{(0,2),(0,4),(1,3),(2,2)\}$.
\end{proposition}

\begin{proof}
By Lemma \ref{NS n1+n2>6}, we have that $\rho(Y)=1$ whenever $n_1+n_2\ge 6$.

Assume $n_1+n_2=4$. The resolution $\widetilde{Y}$ of $Y$ is a rational surface with Picard number
$\rho(\widetilde{Y})=\chi(\O_{\widetilde{Y}})-2=8$.
Since $\widetilde{Y}$ is obtained by blowing up $Y$ at $n_1n_2$ points, we have
$\rho(\widetilde{Y})=\rho(Y)+n_1n_2$.
Hence $\rho(Y)=8$ for $(n_1,n_2)=(0,4)$, $\rho(Y)=5$ for $(n_1,n_2)=(1,3)$, and $\rho(Y)=4$ for $(n_1,n_2)=(2,2)$.

Finally, suppose that $n_1+n_2=2$. When $(n_1,n_2)=(0,2)$, $Y$ is isomorphic to $\P^1\times\P^1$, then $\rho(Y)=2$. When $(n_1,n_2)=(1,1)$, $Y$ is a quadratic cone in $\P^3$, then $\rho(Y)=1$.
\end{proof}

\begin{proof}[Proof of Theorem~\ref{Main Theorem}] By Corollary~\ref{Cor Picard}, $S$ has Picard number bigger than $1$ if and only if some intermediate quotient has Picard number bigger than $1$. Proposition~\ref{Picard intermediate} implies that some intermediate quotient has Picard number bigger than $1$ if and only if some permutation of $(n_1,n_2,n_3)$ belongs to
\[
\{(0,2,2n)\colon n\ge 1\}\cup\{(0,4,2n)\colon n\ge 2\}\cup\{(1,3,2n+1)\colon n\ge 0\}\cup \{(2,2,2n)\colon n\ge 1\}.
\]
This yields the desired result.
\end{proof}


\section{Ulrich line bundles on general bidouble Planes}\label{section4}

In this section, we determine the range of branch degrees for which Ulrich line bundles could exist.  
The strategy is to analyze the existence of Ulrich line bundles on smooth bidouble planes by distinguishing two complementary cases, as dictated by Theorem~\ref{Main Theorem}.  
First, we establish non-existence results for covers branched in sufficiently high degree, and then we turn to the low-degree cases where existence may occur. Throughout this section, we denote by $H$ the line bundle $\pi^*\mathcal{O}_{\mathbb{P}^2}(1)$.

\subsection{Non-existence of Ulrich line bundle in high degree} 
We begin by showing that, for bidouble covers of $\P^2$ branched along divisors of large degree, Ulrich line bundles cannot exist. 

\begin{lemma}\label{odd line bundles}
Let $\pi\colon S\to \P^2$ be a smooth odd bidouble plane. Then, $S$ admits no Ulrich vector bundles of odd rank with respect to $H$.
\end{lemma}

\begin{proof}
Let $n_1$, $n_2$, and $n_3$ be the ramification degrees of $\pi$, and let 
\[
n=n_1+n_2+n_3.
\]
Since each $n_i$ is odd, the sum $n$ is also odd. By Proposition \ref{invariants}, we have $2K_S = (n-6)H$. For the sake of contradiction, assume that an odd rank Ulrich bundle $\mathcal{L}$ on $S$ exists with respect to $H$. By Proposition~\ref{iff Ulrich}, in particular, we have
\begin{equation*}
c_1(\mathcal{L})\cdot H \;=\; \frac{\operatorname{rk}(\L)}{2}(3H+K_S)\cdot H.
\end{equation*}
Using $H^2=4$ and $K_S\cdot H=2(n-6)$, and multiplying both sides by $n-6$, we obtain 
\begin{align*}
2c_1(\mathcal{L})\cdot K_S &= \frac{\operatorname{rk}(\L)(n-6)}{2}(3H^2+K_S\cdot H)\\
&=n(n-6)\operatorname{rk}(\L).
\end{align*}
Since the left-hand side is an integer, the right-hand side must also be an integer. But $n(n-6)\operatorname{rk}(\L)$ is odd, while $2\,c_1(\mathcal{L})\cdot K_S$ is even, which is a contradiction. Thus, no odd rank Ulrich bundle can exist. \end{proof}

\begin{lemma}\label{non-existence} 
Let $\pi\colon S\to \P^2$ be a smooth even bidouble plane that admits an Ulrich line bundle with respect to $H$. Then
the Picard number $\rho(S)$ satisfies $\rho(S) > 1$. 
\end{lemma}
\begin{proof}
Let $n_1$, $n_2$, and $n_3$ be the ramification degrees of $\pi$. Assume to the contrary that $S$ admits an Ulrich line bundle $\L$ and that $\rho(S)=1$. 
Then $\pi^{*}(\Pic(\P^2))$ is a finite-index subgroup of $\Pic(S)$, so there exist line bundles $\L'\in\Pic(S)$ and $\T\in \Pic(S)_{tor}$ integers $a$ and $m\ge1$, with $\gcd(a,m)=1$ such that $\L\cong \L'\otimes\T$ and $m\L'=\O_{S}(a)$. Notice that $\L$ and $\L'$ are numerically equivalent. Since 
$$
c_1(\L)\cdot H=\frac{4a}{m},
$$ 
and $\gcd(a,m)=1$, it follows that $m$ divides $4$. Thus, $m\in \{1,2,4\}$. We claim that $m=1$. 
    
From $K_S\cdot H=2(n_1+n_2+n_3)-12$ and Proposition~\ref{iff Ulrich}, we obtain
    \begin{equation}\label{H.L}
    \frac{4a}{m}=c_1(\mathcal{L})\cdot H \;=\; \tfrac{1}{2}(3H+K_S)\cdot H=n_1+n_2+n_3.
    \end{equation}
If $m= 4$, then $a=n_1+n_2+n_3$, an even number, contradicting $\gcd(a,m)=1$.

Now assume $m=2$. Proposition~\ref{iff Ulrich} gives
    \begin{equation}\label{2 condition Ulrich}
    0=\frac{1}{2}
    (c_1^2(\L)-c_1(\L)\cdot K_S)+\chi(\mathcal{O}_S)-H^2.
    \end{equation}
Hence $c_1^2(\L)-c_1(\L)\cdot K_S$ is even. 

Since 
\begin{equation}\label{L^2 and L.KS}
c_1^2(\L)=\frac{4a^2}{m^2}\qquad \text{and}\qquad  c_1(\L)\cdot K_S=\frac{2a}{m}(n_1+n_2+n_3-6),    
\end{equation}
we obtain
    \[
    c_1^2(\L)-c_1(\L)\cdot K_S=a^2-a(n_1+n_2+n_3-6).
    \]
    Because $n_1+n_2+n_3$ is even, the above expression is even only if $a$ is even, again contradicting $\gcd(a,m)=1$. Thus $m\ne 2$, and therefore $m=1$.

We conclude that any Ulrich line bundle $\L$ must be numerically equivalent to some divisor $\O_S(a)$ for some integer $a$. Moreover, \eqref{H.L} yields $4a=n_1+n_2+n_3$. Using \eqref{L^2 and L.KS} with 
$m=1$ and the expression for $\chi(\O_S)$ given in Proposition~\ref{invariants}, the Ulrich condition \eqref{2 condition Ulrich} becomes
\[
0=\frac{1}{2}\left(3(n_1+n_2+n_3)-\frac{(n_1+n_2+n_3)^2}{4}\right)+\frac{1}{4}\left((n^2_1+n^2_2+n^2_3)+n_1n_2+n_1n_3+n_2n_3-6(n_1+n_2+n_3)\right),
\]
hence $n_1^2+n_2^2+n_3^2=0$, which is a contradiction. Therefore $\rho(S)\neq 1$, hence $\rho(S)>1$.
\end{proof}

Theorem \ref{Main Theorem}, Lemma \ref{odd line bundles}, and Lemma \ref{non-existence} have the following consequence.
\begin{corollary}\label{Cor Ulrich}
Let $S\to \mathbb{P}^2$ be a smooth bidouble plane branched along general curves $D_1,D_2,D_3$ of degrees $n_1,n_2,n_3$. If no permutation of $(n_1,n_2,n_3)$ lies in
\[
\{(0,2,2n):n\geq 1\}\;\cup\;\{(0,4,2n):n\geq 2\}\;\cup\;\{(2,2,2n):n\geq 1\},
\]
then $S$ admits no Ulrich line bundles with respect to $H$.
\end{corollary}
 

\subsection{Ulrich line bundles in low degree}
We now turn to the complementary situation where at least one pair of branch degrees is small.  
In this range, the Picard number can jump, and one might hope for the existence of Ulrich line bundles.  

For double covers of $\P^2$, it was shown in \cite{PN21} that the only obstruction is having Picard number $1$.  
In contrast, we will exhibit families of bidouble covers with Picard number greater than one that nonetheless admit no Ulrich line bundles.  
At the same time, we also identify particular low-degree configurations in which Ulrich line bundles exist.

\begin{proposition}\label{(0,2,2n)}
Let $\pi\colon S\to \P^2$ be a bidouble plane branched along the union of the zero divisor, a smooth conic, and a smooth curve of degree 
$2n$ for some $n\ge 3$, all in general position. Then $S$ admits no Ulrich line bundle with respect to $H$.

\end{proposition}
\begin{proof}
Notice that there is a morphism $\pi_1\colon S\to \P^1\times\P^1$, which is a double cover branched along the divisor $(2n,2n)$. For the sake of contradiction, suppose that there exists an Ulrich line bundle $\L$ with respect to $H$. By Proposition~\ref{iff Ulrich}, we have
    \begin{align}
c_1(\mathcal{L})\cdot H \;&=\; \tfrac{1}{2}(3H+K_S)\cdot H. \label{Ulrich conditions}\\
    0&=\frac{1}{2}
    (c_1^2(\L)-c_1(\L)\cdot K_S)+\chi(\mathcal{O}_S)-H^2. \label{Ulrich conditions2}        
    \end{align}
By Proposition \ref{invariants}, $\chi(\O_S)=n^2-2n+2$, $K_S\cdot H=4(n-2)$, and $H^2=4$.
    The Lemma \ref{NS} and Proposition \ref{Picard intermediate} imply $\rho(S)=\rho(\P^1\times\P^1)$, so Remark~\ref{equal rho} implies that there exists an integer $m\in \{1,2\}$ such that $m\L$ is the pullback of a class in $\NS(\P^1\times\P^1)$. Then, there exist integers $a$ and $b$ with 
    $$
    m\L=\pi_1^*(\O_{\P^1\times\P^1}(a,b)),
    $$ 
    and consequently, we have that 
    \[
    c_1(\L)\cdot H=\frac{2}{m}(a+b),\quad c_1(\L)\cdot K_S=\frac{2}{m}(n-2)(a+b), \quad \text{and }\quad c_1^2(\L)=\frac{4ab}{m^2}.
    \]
    Equation~\eqref{Ulrich conditions} yields $a+b=(n+1)m$, and \eqref{Ulrich conditions2} gives
    \begin{align*}
    0= \frac{1}{2}
    \left(\frac{4ab}{m^2}-\frac{2(n-2)(a+b)}{m}\right)+(n^2-2n+2)-4=\frac{2ab}{m^2}-n.
    \end{align*}
    In particular, $2ab=nm^2$. Substituting $b=(n+1)m-a$, we have the following quadratic equation in $a$
    \[
    2a^2-2m(n+1)a+m^2n=0.
    \]
    The discriminant of this equation is $4m^2(n^2+1)$, which is never a perfect square for $n\ge 1$. This contradicts the integrality of $a$.
\end{proof}


\begin{theorem}[Ulrich line bundles]\label{line bundles L}
Let $\pi\colon S\to \mathbb{P}^2$ be a bidouble plane branched along general curves $D_1,D_2,D_3$ of degrees $n_1,n_2,n_3$. If $S$ admits an Ulrich line bundle with respect to $H$,
then some permutation of $(n_1,n_2,n_3)$ belongs to
\[
\{(0,2,2)\}\;\cup\;\{(0,4,2n):n\geq 1\}\;\cup\;\{(2,2,2n):n\geq 1\}.
\]
\end{theorem}

\begin{proof}
Corollary~\ref{Cor Ulrich} shows that if a bidouble cover 
$S$ branched along curves of degrees $(n_1,n_2,n_3)$ admits an Ulrich line bundle, then some permutation of $(n_1,n_2,n_3)$ belongs to
\[
\{(0,2,2n)\colon n\ge 1\}\;\cup\;\{(0,4,2n)\colon n\ge 2\}\;\cup\;
\{(2,2,2n)\colon n\ge 1\}.
\] 
By Proposition~\ref{(0,2,2n)}, if a permutation of $(n_1,n_2,n_3)$ is equal to $(0,2,2n)$ with $n\ge 3$, then $S$ admits no Ulrich line bundles with respect to $H$. Therefore, the only possible cases occur when some permutation of $(n_1,n_2,n_3)$ belongs to
\[
\{(0,2,2)\}\;\cup\;\{(0,4,2n)\colon n\ge 1\}\;\cup\;\{(2,2,2n)\colon n\ge 1\},
\]
which is precisely the statement of the theorem. This completes the proof.
\end{proof}

While the converse of Theorem~\ref{line bundles L} remains open, the following proposition provides explicit examples of bidouble planes carrying Ulrich line bundles.

\begin{proposition}\label{(0,2,2) and (0,2,4)}
Let $\pi\colon S \to \mathbb{P}^2$ be a bidouble plane branched along general curves $D_1,D_2,D_3$ of degrees $n_1,n_2,n_3$ with $n_1\le n_2\le n_3$. If $(n_1,n_2,n_3)$ is 
$(0,2,2)$ or $(0,2,4)$, $S$ admits an Ulrich line bundle
with respect to $H$.
\end{proposition}

\begin{proof}
We treat the two cases separately.

\smallskip
\noindent\textbf{Case (0,2,2).} In this case, $S$ is a del Pezzo surface of degree $4$, and $-K_S = H$. By \cite[Theorem~1.1]{PonsTonini}, every del Pezzo surface of degree $d$ admits an Ulrich line bundle with respect to $-K_X$. More precisely, a line bundle $\mathcal{L}$ is Ulrich if and only if $\mathcal{L}\cong \mathcal{O}_S(D)$ for some rational normal curve $D \subset S$ of degree at most $d$. Therefore, $S$ admits an Ulrich line bundle with respect to $H$.

\smallskip
\noindent\textbf{Case (0,2,4).} Let $\pi\colon S\to\P^2$ be the bidouble plane ramified along the zero divisor, a smooth conic $C$, and a smooth quartic $Q$. Then $S$ is a K3 surface, and following the notation in the diagram \eqref{diagram}, it admits two intermediate double planes: $Y_3$, the double plane branched along $C$, and $Y_2$, the double plane branched along $Q$. The surface $Y_3$ is isomorphic to $\P^1\times\P^1$, hence $\rho(Y_3)=2$. The surface $Y_2$ is a del Pezzo surface of degree $2$, so $\rho(Y_2)=8$, and we may view $Y_2$ as the blow–up of $\P^2$ at $7$ general points. In particular,
\[
\NS(Y_2) \supset \langle H_{Y_2}, E_1,\dots,E_7\rangle,
\]
where $H_{Y_2}$ is the pullback of a line in $\P^2$ via $\pi_2$, and $E_i$ are the exceptional divisors.

By Proposition~\ref{iff Ulrich}, a line bundle $\L$ on $S$ is Ulrich with respect to $H$ if and only if
\begin{equation}\label{num conditions 024}
c_1^2(\L) = 4,\qquad c_1(\L)\cdot H= 6,\qquad
h^0(S,\L(-H)) = h^0(S,\L^\vee(2H)) = 0.
\end{equation}
We now construct such a line bundle. Let $\ell\subset \P^2$ be a line tangent to $C$ which intersects $Q$ in $4$ distinct points. By \cite[Theorem~4.4]{PN21}, the preimage $\pi^{-1}(\ell) = \Gamma_1 + \Gamma_2$
is the union of two smooth curves of genus $1$. In particular, $\Gamma_i^2 = 2g(\Gamma_i) - 2 = 0$.
Moreover, $\pi^*\ell \sim H$, hence
$H = \Gamma_1 + \Gamma_2$,\
and by symmetry, we obtain
$H\cdot \Gamma_j = 2$ for each $j$.

Let $E_i' := p_2^*E_i$ denote
the pullback of the exceptional divisors from $Y_2$ to $S$. Since $H_{Y_2}= -K_{Y_2}$, we have
$H = p_2^*(H_{Y_2}) = p_2^*(-K_{Y_2})$.
In particular,
\[
H\cdot E_i' = p_2^*(-K_{Y_2})\cdot p_2^*E_i
= 2(-K_{Y_2}\cdot E_i) = 2
\]
for all $i$. Since $H=\Gamma_1+\Gamma_2$ we have
\[
\Gamma_1\cdot E_i' = \Gamma_2\cdot E_i' = 1
\]
for all $i$.

Choose two different exceptional divisors $E_1,E_2$ on $Y_2$ and consider the divisor
\[
D := H + \Gamma_1 + E_1' - E_2' \in \Div(S).
\]

First, we verify that the numerical conditions in \eqref{num conditions 024} are satisfied. Using the intersection numbers above and $(E_i')^2 = 2E_i^2 = -2$, we compute
\[
\begin{aligned}
D\cdot H= (H+\Gamma_1+E_1'-E_2')\cdot H = H^2 + \Gamma_1\cdot H + E_1'\cdot H - E_2'\cdot H = 4 + 2 + 2 - 2 = 6,
\end{aligned}
\]
and
\[
\begin{aligned}
D^{2}
&= H^2 + \Gamma_1^2 + (E_1')^2 + (E_2')^2 + 2\big(H\cdot\Gamma_1 + H\cdot E_1' - H\cdot E_2'
       + \Gamma_1\cdot E_1' - \Gamma_1\cdot E_2' - E_1'\cdot E_2'\big) \\
&= 4 + 0 - 2 - 2 + 2\big(2 + 2 - 2 + 1 - 1 - 0\big)= 4.
\end{aligned}
\]
Thus $D^{2} = 4$ and $D\cdot H = 6$. 

We now show that the cohomological vanishings in \eqref{num conditions 024} hold. Due to $H = \Gamma_1+\Gamma_2$, we have
\[
D - H = \Gamma_1 + E_1' - E_2', \qquad
2H - D = \Gamma_2 - E_1' + E_2'.
\]

For the sake of clarity, we show that $h^0(S,D-H)=0$; the vanishing $h^0(S,2H-D)=0$ is proved similarly. For a contradiction, assume that $h^0(S,D-H)>0$, therefore, there exists an effective divisor $F\sim D-H$.
A direct computation gives $F^2 = -4$ and $H\cdot F = 2$.
Moreover,
\[
F\cdot E_1' = \Gamma_1\cdot E_1' + (E_1')^2 - E_2'\cdot E_1' = 1 - 2 - 0 = -1.
\]
Since $F$ is effective, the inequality $F\cdot E_1'<0$ forces $E_1'$ to be a component of $F$, so we could write $F = E_1' + F'$, with $F'$ effective. However, notice that
$(F')^2 = (\Gamma_1 - E_2')^2 = -4$ and
$H\cdot F' = H\cdot\Gamma_1 - H\cdot E_2' = 0$.
Since $H$ is ample, any nonzero effective divisor $Z$ must satisfy $H\cdot Z > 0$, hence $F'$ cannot be effective. This contradiction shows $h^0(S,D-H) =  0$.

Putting everything together, the line bundle $\L := \O_S(D)$ satisfies
\[
c_1^2(\L)=4,\qquad H\cdot c_1(\L)=6,\qquad h^0(S,\L(-H))=h^0(S,\L^\vee(2H))=0,
\]
so, by Proposition~\ref{iff Ulrich}, $\L$ is an Ulrich line bundle with respect to $H$.
\end{proof}

\begin{remark}
In case $(2,2,2)$, \cite[Theorem~4.8]{Gar22} implies that $S$ is an Enriques surface of degree $4$. In this case, $S$ is nodal, and therefore the methods of \cite{BorisoNuer} do not apply.

In case $(0,4,4)$, one can construct a divisor analogous to the one used in case $(0,2,4)$, which satisfies the numerical conditions of
Proposition~\ref{iff Ulrich}. However, verifying the required cohomological
vanishings appears to be substantially more delicate, and we were unable to
establish them. More precisely, instead of choosing a line tangent to the conic as in the case
$(0,2,4)$, one may consider a smooth conic whose intersection with one of the
quartic branch curves is an even divisor. The existence of such a conic follows
from \cite[Corollary~1.3]{PN21}. Nevertheless, this approach does not seem to lead
to a proof of the necessary vanishing conditions.

For the remaining cases of Theorem~\ref{line bundles L} that could
a priori admit an Ulrich line bundle, Lemma~\ref{NS} shows that
$\rho(S)=\rho(Y_3)$. This strongly restricts the geometry of $\Pic(S)$: in
particular, for any line bundle $\mathcal L\in\Pic(S)$, a suitable multiple of
$\mathcal L$ is the pullback of a line bundle on $Y_3$. This observation suggests
that the existence of Ulrich line bundles on $S$ in these cases is unlikely, as illustrated by Proposition~\ref{(0,2,2n)}.
\end{remark}


\end{document}